\documentclass[reqno,11pt]{amsart}

\usepackage{hyperref}
\usepackage{amscd, amssymb , amsmath,epsfig,subfig,fullpage}
\usepackage{cancel}
\usepackage{mathrsfs}
\usepackage{epsfig}
\usepackage{color}
\definecolor{brightlavender}{rgb}{0.75, 0.58, 0.89}
\definecolor{carnelian}{rgb}{0.7, 0.11, 0.11}
\usepackage[all]{xy}
\usepackage{verbatim}

\newcommand{\ptr}{\operatorname{ptr}}
\newcommand{\tr}{\operatorname{tr}}
\newcommand{\et}{\quad\text{and}\quad} 
\newtheorem{Df}{Definition}
\newtheorem{definition}[Df]{Definition}
\newtheorem{theorem}[Df]{Theorem}

\newtheorem{proposition}[Df]{Proposition}
\newtheorem{lemma}[Df]{Lemma}

\newtheorem{remark}[Df]{Remark}

\newtheorem{corollary}[Df]{Corollary}
\newtheorem{conjecture}[Df]{Conjecture}
\newcommand{\tcoev}{\stackrel{\longleftarrow}{\operatorname{coev}}}
\newcommand{\tev}{\stackrel{\longleftarrow}{\operatorname{ev}}}
\newcommand{\ev}{\stackrel{\longrightarrow}{\operatorname{ev}}}
\newcommand{\coev}{\stackrel{\longrightarrow}{\operatorname{coev}}}

\newcommand{\diag}{\operatorname{diag}}
\newcommand{\bp}[1]{{\left(#1\right)}}

\newcommand{\oo}{{\infty}}
\newcommand{\ro}{{r}}
\newcommand{\rk}{{n}}
\newcommand{\Uqg}{\ensuremath{\mathcal U}}
\newcommand{\UqgH}{\ensuremath{\Uqg^{H}}}

\newcommand{\cat}{\mathscr{C}}

\newcommand{\Hh}{H\hspace{-1ex}\,H^h}

\newcommand{\La}{{L_W}}
\newcommand{\Proj}{{\mathsf{Proj}}}

\newcommand{\Int}{\operatorname{Int}}
\newcommand{\I}{\mathcal{I}}

\newcommand{\ang}[1]{{\left\langle{#1}\right\rangle}}

\renewcommand{\SS}{\Sigma}

\newcommand{\Ffun}{\ensuremath{\mathsf{F}}}

\newcommand{\qt}{\operatorname{\mathsf{t}}}

\newcommand{\ideal}{\I}

\newcommand{\wtb}{\mu}
\newcommand{\h}{\ensuremath{\mathfrak{h}}}
\newcommand{\roots}{\Delta}
\newcommand{\e}{\operatorname{e}}
\newcommand{\wb}{\overline}
\newcommand{\wt}{\widetilde}
\newcommand{\wh}{\widehat}
\newcommand{\Spec}{\operatorname{Spec}}
\newcommand{\Graph}{\operatorname{Gr}}

\newcommand{\Gr}{\ensuremath{\mathcal{G}}}
\newcommand{\D}{\ensuremath{\mathcal{D}}}
\newcommand{\X}{\ensuremath{\mathcal{X}}}

\newcommand{\C}{\ensuremath{\mathbb{C}}}
\newcommand{\Z}{\ensuremath{\mathbb{Z}}}
\newcommand{\R}{\ensuremath{\mathbb{R}}}
\newcommand{\N}{\ensuremath{\mathbb{N}}}
\newcommand{\Uhg}{\ensuremath{U_{h}(\g)}}
\newcommand{\g}{\ensuremath{\mathfrak{g}}}
\newcommand{\slt}{\ensuremath{\mathfrak{sl}(2)}}

\newcommand{\M}{\ensuremath{\mathcal{M}}}

\newcommand{\End}{\operatorname{End}}
\newcommand{\Hom}{\operatorname{Hom}}
\newcommand{\Aut}{\operatorname{Aut}}

\newcommand{\unit}{\ensuremath{\mathbb{I}}}
\newcommand{\Id}{\operatorname{Id}}
\newcommand{\qdim}{\operatorname{qdim}}
\newcommand{\FK}{{\Bbbk}}
\newcommand{\qn}[1]{{\left\{#1\right\}}}
\newcommand{\qN}[1]{{\left[#1\right]}}
\newcommand{\qd}{\operatorname{\mathsf{d}}}

\newcommand{\A}{\ensuremath{\mathsf{A}}}

\newcommand{\sll}{\mathfrak{sl}}

\newcommand{\LL}{\mathcal{L}}

\newcommand{\bb}{\operatorname{\mathsf{b}}}

\newcommand{\T}{{\mathcal T}}

\newcommand{\epsh}[2]
         {\begin{array}{c} \hspace{-1.3mm}
        \raisebox{-4pt}{\epsfig{figure=#1,height=#2}}
        \hspace{-1.9mm}\end{array}}

\newcounter{exo} \newcounter{numexercice}
\renewcommand{\theexo}{\arabic{exo}}

\begin{document}

\title
{The trace on projective representations of quantum groups}

\author[N. Geer]{Nathan Geer}
\address{Mathematics \& Statistics\\
  Utah State University \\
  Logan, Utah 84322, USA}
\email{nathan.geer@gmail.com}

\author[B. Patureau-Mirand]{Bertrand Patureau-Mirand}
\address{Univ. Bretagne - Sud,  UMR 6205, LMBA, F-56000 Vannes, France}
\email{bertrand.patureau@univ-ubs.fr}

\date{\today}

\begin{abstract}
For certain roots of unity, we consider the categories of weight modules over three quantum groups:  small, un-restricted and unrolled.  The first main theorem of this paper is to show that there is a modified trace on the projective modules of the first two categories.  The second main theorem is to show that category over the unrolled quantum group is ribbon.  Partial results related to these theorems were known previously.    
\end{abstract}

\maketitle
\setcounter{tocdepth}{3}

\section*{Introduction}
For an odd ordered root of unity $\xi$ and lattice $L$, let
$\Uqg_\xi^L$, $\UqgH_\xi$ and $\wb\Uqg_\xi^L$ be the unrestricted,
unrolled and small quantum groups, respectively (see Section
\ref{S:RepQG}).  Let $\cat_{odd}$ (resp. $\cat^H_{odd}$,
resp. $\wb\cat_{odd}$) be the category of $\Uqg_\xi^L$ (resp. $\UqgH$,
resp. $\wb\Uqg_\xi^L$) weight modules.
The usual construction of quantum invariants do not directly apply to
these categories because of the following obstructions: the categories
are not semi-simple and have vanishing quantum dimensions.  Partial
results overcoming these obstructions have been obtained in
\cite{CGP3, GKP1, GKP2, GP3, GPT1, GPV}.  In this paper we generalize
some of these results using a new concept called generically
semi-simple (loosely meaning the category is graded and semi-simple on
a dense portion of the graded pieces).  In \cite{BBG, BGPR}, the
results of this paper will be used in future work to construct
topological invariants.

The first part of this paper contains a general theory with two main
theorems which extend properties of generic simple objects to general
properties in the full category.  The first (Theorem~\ref{T:GenTwist})
loosely says that if a category is generically semi-simple, pivotal
and braided such that there is a twist for every generic simple object
then the full category has a twist and so the category is ribbon.  The
second (Theorem \ref{T:GenericDimTrace}) loosely says that if a
category is generically semi-simple and pivotal with a right trace on
its projective objects whose modified dimension satisfies
$\qd(V)=\qd(V^*)$ for all generic simple objects $V$ then the right
trace is a (two sided) trace.

We apply these theorems to the categories of modules over the different
quantum groups mentioned above.  It was previously known that
$\cat^H_{odd}$ is a braided pivotal category and $\cat_{odd}$ is a
pivotal category with a right trace, see \cite{GP3} and \cite{GKP2},
respectively.  
Here we remark that $\cat^H_{odd}$ and $\cat_{odd}$ are generically
semi-simple (see Subsection \ref{SS:SemiSimple}).  In Subsection
\ref{SS:RibbonStr} we show that each generic simple module in
$\cat^H_{odd}$ has a twist and thus Theorem~\ref{T:GenTwist} implies
$\cat^H_{odd}$ is ribbon.  Subsection \ref{SS:modifiedDim} contains a
proof that the modified dimension satisfies $\qd(V)=\qd(V^*)$ for all
generic simple objects $V$ of $\cat_{odd}$ and so Theorem
\ref{T:GenericDimTrace} implies the unique right trace (up to global
scalar) on the ideal of projective modules of $\cat_{odd}$ is a trace.
This unique trace (up to global scalar) induces a trace on
$\wb\cat_{odd}$.

The main theorems in this paper about $\cat^H_{odd}$ and $\cat_{odd}$
use deep results developed by De Concini, Kac, Procesi, Reshetikhin
and Rosso in the series of papers \cite{DK,DKP,DKP2,DPRR}.  In
particular, we use the quantum co-adjoint action.  In general, these
results hold for odd ordered roots of unity.  However, in Section
\ref{SS:ccl} we show directly that in the case of $\slt$ the results
discussed above hold for even and odd ordered roots of unity.  In
particular, we show $\cat_{\slt}$ is braided and $\cat^H_{\slt}$ and
$\wb\cat_{\slt}$ have unique traces.  Finally, in Section
\ref{S:ConjEven} we conjecture that Subsections \ref{SS:RibbonStr} and
\ref{SS:modifiedDim} generalize to even ordered root of unity.

\section{Tensor categories}\label{S:SphPsi}
In this section we give the basic definitions for tensor categories
and the functors induced by graphic calculus.

\subsection{Pivotal and ribbon categories}\label{SS:LinearCat}
We recall the definition of pivotal and ribbon tensor categories, for more details see
for instance, \cite{BW}.  In this paper, we consider strict tensor
categories with tensor product $\otimes$ and unit object $\unit$.  Let
$\cat$ be such a category.  
The notation $V\in \cat$
means that $V$ is an object of $\cat$.
 
The category $\cat$ is a \emph{pivotal category} if it has duality morphisms 
$$\coev_{V} : \:\:
\unit \rightarrow V\otimes V^{*} , \quad \ev_{V}: \:\: V^*\otimes V\rightarrow
\unit , \quad \tcoev_V: \:\: \unit \rightarrow V^{*}\otimes V\quad {\rm {and}}
\quad \tev_V: \:\: V\otimes V^*\rightarrow \unit$$ which satisfy compatibility
conditions (see for example \cite{BW, GKP2}).   In particular, the left dual and right dual of a morphism $f\colon V \to W$ in $\cat$ coincide:
\begin{align*}
f^*&= (\ev_W \otimes  \Id_{V^*})(\Id_{W^*}  \otimes f \otimes \Id_{V^*})(\Id_{W^*}\otimes \coev_V)\\
 &= (\Id_{V^*} \otimes \tev_W)(\Id_{V^*} \otimes f \otimes \Id_{W^*})(\tcoev_V \otimes \Id_{W^*}) \colon W^*\to V^*.
\end{align*}
Then there is a natural notion of right categorical (partial) trace in $\cat$: for any $V,W\in\cat$,
$$
\begin{array}{rcl}
  \tr_R:\End_\cat(V)&\to&\FK\\
  f&\mapsto&  \tev_V(f \otimes \Id_{V^*})\coev_V
\end{array}\et$$$$
\begin{array}{rcl}
\ptr_R:\End_\cat(V\otimes W)&\to&\End_\cat(V)\\
f&\mapsto & (\Id_V \otimes \tev_W)(f\otimes \Id_{W^*})(\Id_V \otimes \coev_W)
\end{array}
$$
and analogous left categorical (partial) trace $\tr_L$
(resp. $\ptr_L$) defined using $\ev$ and $\tcoev$.
\\

A \emph{braiding} on $\cat$ consists of a family of natural
isomorphisms $\{c_{V,W}: V \otimes W \rightarrow W\otimes V \}$
satisfying the Hexagon Axiom:
$$c_{U,V\otimes W}=(\Id_V\otimes c_{U,W})\circ(c_{U,V}\otimes\Id_W)\et
c_{U\otimes V,W}=(c_{U,W}\otimes \Id_V)\circ(\Id_U\otimes c_{V,W})$$
for all $U,V,W\in\cat$.  We say $\cat$ is \emph{braided} if it has a
braiding.  If $\cat$ is pivotal and braided, one can define a family
of natural automorphisms
$$\theta_V=\ptr_R(c_{V,V}):V\to V.$$
We say that $\cat$ is \emph{ribbon} and the morphism $\theta$ is a
\emph{twist} if
\begin{equation}
  \label{eq:twistc}
  \theta_{V^*}=(\theta_V)^*
\end{equation}
for all $V\in\cat$.
\begin{remark}
  An equivalent definition of a ribbon category is a braided
  left rigid balanced category.  Left rigid means that there are left
  duals $(V,V^*,\ev_V,\coev_V)$ and the balance $\theta$ is a natural
  automorphism of the identity functor satisfying $\theta_{U\otimes
    V}=(\theta_{U}\otimes \theta_V)\circ c_{V\otimes U}\circ c_{U\otimes
    V}$ and Equation~\eqref{eq:twistc}.  
\end{remark}

\subsection{$\FK$-categories}
Let $\FK$ be a integral domain.  
A \emph{$\FK$-category} is a category $\cat$ such that its hom-sets are left $\FK$-modules, the composition of morphisms is $\FK$-bilinear, and the canonical $\FK$-algebra map $\FK \to \End_\cat(\unit), k \mapsto k \, \Id_\unit$ is an isomorphism.
A \emph{tensor $\FK$-category} is a tensor category $\cat$ such that $\cat$ is a $\FK$-category and the tensor product of morphisms is $\FK$-bilinear.  
An object $V$ of $\cat$ is \emph{simple} if $\End_\cat(V)= \FK \Id_V$.
Let $V$ be an object in $\cat$ and let $\alpha:V\to W$ and $\beta:W\to V$ be morphisms.  The triple
$(V,\alpha,\beta)$ (or just the object $V$) is a \emph{retract} of $W$ if
$\beta\alpha=\Id_{V}$.  An object $W$ is a \emph{direct sum} of the finite
family $(V_i)_i$ of objects of $\cat$ if there exist retracts
$(V_i,\alpha_i,\beta_i)$ of $W$ with $\beta_i\alpha_j=0$ for $i\neq j$
and $\Id_W=\sum_i\alpha_i\beta_i$.
An object which is a direct sum of simple objects is called {\em
  semi-simple}.

\subsection{Traces on ideals in pivotal categories}\label{SS:trace}
Here we recall the definition of a (right) trace on an (right) ideal in a pivotal $\FK$-category $\cat$, for more details see \cite{GPV}.   By a \emph{right ideal} of  $\cat$ we mean a full subcategory $\ideal$ of $\cat$ such that:  
\begin{enumerate}
\item  If $V\in \ideal$ and $W\in \cat$, then $V\otimes W \in \ideal$.
\item If $V\in \ideal$ and if $W\in \cat$ is a retract of $V$,
  then $W\in \ideal$.  
\end{enumerate}
One defines similarly the notion of a \emph{left ideal} by replacing in the above definition $V\otimes W \in \ideal$ by  $W\otimes V \in \ideal$.
A full subcategory $\ideal$ of $\cat$ is an \emph{ideal} if it is both a
right and left ideal.

If $\ideal$ is a right ideal in $\cat$ 
then a \emph{right trace} on $\ideal$ is a family of linear functions
$$\{\qt_V:\End_\cat(V)\rightarrow \FK \}_{V\in \ideal}$$
such that following two conditions hold:
\begin{enumerate}
\item  If $U,V\in \ideal$ then for any morphisms $f:V\rightarrow U $ and $g:U\rightarrow V$  in $\cat$ we have 
\begin{equation*}
\qt_V(g f)=\qt_U(f  g).
\end{equation*} 
\item \label{I:Prop2Trace} If $U\in \ideal$ and $W\in \cat$ then for any $f\in \End_\cat(U\otimes W)$ we have
\begin{equation*}
\qt_{U\otimes W}\bp{f}=\qt_U \bp{\ptr_R(f)}.
\end{equation*}
\end{enumerate}
The notion of a left trace on a left ideal is obtained by
replacing \eqref{I:Prop2Trace}  in the above definition with $\qt_{W\otimes U}\bp{f}=\qt_U \bp{\ptr_L(f)}$ for all $ f\in \End_\cat(W\otimes
U)$.  A family
$\qt=\{\qt_V \}_{V\in \ideal}$ is a \emph{trace} if $\ideal$ is an
ideal and $\qt$ is both a left and right trace.

The class of projective modules $\Proj$ in $\cat$ is an ideal.   In a pivotal category projective and injective objects coincide (see \cite{GPV}).  The ideal $\Proj$ is an important example which we will consider later in this paper.
\subsection{Colored ribbon graph invariants}
Let $\cat$ be a pivotal category.  A morphism
$f:V_1\otimes{\cdots}\otimes V_n \rightarrow W_1\otimes{\cdots}\otimes
W_m$ in $\cat$ can be represented by a box and arrows:
$$ 
\xymatrix{ \ar@< 8pt>[d]^{W_m}_{... \hspace{1pt}}
  \ar@< -8pt>[d]_{W_1}\\
  *+[F]\txt{ \: \; $f$ \; \;} \ar@< 8pt>[d]^{V_n}_{... \hspace{1pt}}
  \ar@< -8pt>[d]_{V_1}\\
  \: }
$$
Boxes as above are called {\it coupons}. 
By a ribbon graph in an
oriented manifold $\SS$, we mean an
oriented compact surface embedded in $\SS$ which decomposed into
elementary pieces: bands, annuli, and coupons (see \cite{Tu}) and is the thickening of an oriented graph.  
In particular, the
vertices of the graph lying in $\Int \SS=\SS\setminus\partial \SS$ are
thickened to coupons.  A $\cat$-colored ribbon graph is a ribbon graph
whose (thickened) edges are colored by objects of $\cat$ and whose
coupons are colored by morphisms of~$\cat$.  
The intersection of a $\cat$-colored ribbon graph in $\Sigma$ with $\partial \Sigma$ is required to be empty or to consist only of vertices of valency~1.  
When $\SS$ is a surface the ribbon graph is just a tubular neighborhood of the graph.

A $\cat$-colored ribbon graph in $\R^2$ 
is
called \emph{planar}. A $\cat$-colored ribbon graph in
$S^{2}=\R^2\cup\{\infty\}$ is called \emph{spherical}.  A $\cat$-colored ribbon graph in $\R^3$ or  $\R^2\times [0,1]$ are called \emph{spatial}. 

For $i\in\{2,3\}$, the $\cat$-colored ribbon graphs in $\R^{i-1}\times [0,1]$  form a category $\Graph^i_\cat$ as follows: objects of
$\Graph^i_\cat$ are finite sequences of pairs $(X,\varepsilon)$, where $X\in
\cat$ and $\varepsilon=\pm$.  Morphisms of $\Graph^i_\cat$ are isotopy classes
of $\cat$-colored ribbon graphs in $\R^{i-1}\times [0,1]$. 
 By a \emph{(1,1)-ribbon graph} in $\Graph^i_\cat$ we mean a $\cat$-colored ribbon graph which is an endomorphism of an object $(V,+)$ in $\Graph^i_\cat$.  
  Let $\Ffun: \Graph^i_\cat \to \cat$ be the Reshetikhin-Turaev $\FK$-linear
functor (see \cite{GPT2}).
If $\cat$ is a ribbon category, the functor on planer graphs $\Ffun: \Graph^2_\cat \to \cat$ extends to the functor on spatial graphs $\Ffun: \Graph^3_\cat \to \cat$.  

\subsection{Renormalized colored ribbon graph invariants}\label{SS:renorm}
The motivation of this paper is to provide the underpinnings for the construction of topological invariants.  With this in mind, 
 in this subsection, we recall the notion of \emph{re-normalized}
colored ribbon graph invariants introduced and studied in
\cite{GP3,GKP1,GPT1, GPT2, GPV}. 
 This subsection is independent of the rest of the paper.  
 The theory of re-normalized
invariants produces non-trivial invariants in some situations when the
standard approaches fail.  In particular, these invariants can be
non-trivial when quantum dimension vanish.  The re-normalized
invariant of closed $\cat$-colored ribbon graph can be computed in
three steps: 1) cut a special edge of a closed $\cat$-colored ribbon
graph, 2) apply $\Ffun$ to the resulting graph to obtain an
endomorphism and 3) apply to the endomorphism a $\FK$-linear
functional to obtain a number.  If this functional has certain
properties then the number is an invariant of the ribbon graph.  Here
is a more precise definition:

 Let $\cat$ be a pivotal (resp., ribbon) category.  Let $\T_{adm}$ be
 a class of planar (resp. spatial) $\cat$-colored (1,1)-ribbon graphs.
 We denote an element of $\T_{adm}$ by $T_V$ where $V$ is the object
 of $\cat$ which colors the open edge of the (1,1)-ribbon graph.  We
 call $V$ the section of $T_V$.  Given such a graph $T_{V}$ the right
 braid closure gives a well defined equivalence class of a closed
 ribbon graph $ \wh {T_{V}}$ in $\R^2$ (resp. in $\R^3$).  Let
 $\LL_{adm}$ be the class of right braid closures of elements of
 $\T_{adm}$.  Let $\qt=\{\qt_V:End_\cat(V)\to \FK\}_V$ be a family of
 linear maps where $V$ runs over all the sections of elements $T_V\in
 \T_{adm}$.  Suppose that $\qt$ satisfies the condition:
\begin{verse}
  If $T_V,T'_W\in \T_{adm}$ such that $ \wh {T_{V}}$ is isotopic to $
  \wh {T'_{W}}$ then $\qt_V \bp{\Ffun(T_{V})}=\qt_W
  \bp{\Ffun(T'_{W})}$.
\end{verse}
We call the function 
$$\Ffun':\LL_{adm}\to\FK \text{ defined by } \Ffun'\bp{\wh{T_V}}=\qt_V\bp{\Ffun(T_V)}
$$
the \emph{re-normalized invariant} associated to $\LL_{adm}$ and $\qt$.   

We will now give some examples of re-normalized invariants. 
\begin{description}
\item[Example 1] \emph{T-ambi pair.}  
Let $\A$ be a class of simple objects in a  pivotal
$\FK$-category $\cat$.  The classes $\T_{adm}$ and $\LL_{adm}$ are formed by
the trivalent ribbon planar graphs whose edges are colored by elements
of $\A$.   The family $\qt=(\qt_V)_{V\in \A}$ is determined by a mapping
$\qd:\A\to\FK^\times$ constant on isomorphism classes of objects.
Then the map $\qt_V$ is determined by
$\qt_V\bp{\lambda\Id_V}=\lambda\qd(V)$.  If $\Ffun'$ is invariant by
isotopy in the sphere $S^2$ (here we consider an isomorphism
$S^2\simeq\R^2\cup\{\oo\}$) then we say that $(\A,\qd)$ is a
\emph{trivalent ambidextrous pair} or \emph{t-ambi} for short.  
A (modified) $6j$-symbol  is the value of a tetrahedron under  $\Ffun'$.  These $6j$-symbols  are the elementary algebraic ingredients of a  re-normalized Turaev-Viro type invariant of $3$-manifolds defined by state sums on triangulations (\cite{GPT2,GP3}). 
\item[Example 2] \emph{Right trace.} Let $\qt$ be a right trace on a right ideal $\I$  in a pivotal $\FK$-category $\cat$, see Subsection \ref{SS:trace}.  Let $\T_{adm}$ be all the
$\cat$-colored (1,1)-ribbon planar graphs whose sections are in $\I$.
Then $\Ffun'$ is a invariant of planar isotopy but in general it is not an
invariant of isotopy in the sphere $S^2$.  Nevertheless, for $V\in\I$
one can set $\qd(V)=\qt_V\bp{\Id_V}$ then Corollary 7 of \cite{GPV} implies that 
for $\A=\{V\in\I:V\text{ is simple, }V^*\in\I,\, \qd(V)=\qd(V^*)\}$,
$(A, \qd)$ is a trivalent-ambidextrous pair and the restriction of
$\Ffun'$ to $\A$-colored graphs is an invariant of isotopy in the
sphere $S^2$.
\item[Example 3]\emph{(Two-sided) trace.} 
 Let $\cat$ be a pivotal (resp. ribbon) $\FK$-category and let
$\qt$ be a trace on an ideal $\I$.  Again, let $\T_{adm}$ be all $\cat$-colored (1,1)-ribbon planar
(resp. spatial) graphs whose sections are in $\I$.  By setting $\qd(V)=\qt_V\bp{\Id_V}$ for  $V\in\I$ then Theorem 5 of \cite{GPV} implies $\Ffun'$ is  an invariant of spherical  (resp. spatial) ribbon graphs.  Moreover, $\LL_{adm}$ is formed by the $\cat$-colored
ribbon graphs with at least one edge in $\I$.
\end{description}

\subsection{$\Gr$-graded and generically $\Gr$-semi-simple
  categories}\label{SS:GSpher}
We now fix a group $\Gr$.  
\begin{definition}Grading:\\ A pivotal $\FK$-category is {\em $\Gr$-graded} if
  for each $g\in \Gr$ we have a non-empty full subcategory $\cat_g$ of
  $\cat$ stable by retract such that
  \begin{enumerate}
  \item  $\cat=\bigoplus_{g\in\Gr}\cat_g$, 
  \item  if $V\in\cat_g$,  then  $V^{*}\in\cat_{g^{-1}}$,
  \item  if $V\in\cat_g$, $V'\in\cat_{g'}$ then $V\otimes
    V'\in\cat_{gg'}$,
  \item  if $V\in\cat_g$, $V'\in\cat_{g'}$ and $\Hom_\cat(V,V')\neq 0$, then
    $g=g'$.  
  \end{enumerate}
\end{definition}
For a subset $\X\subset\Gr$ we say:
\begin{enumerate}
\item $\X$ is \emph{symmetric} if $\X^{-1}=\X$,
\item $\X$ is \emph{small} in $\Gr$ if the group $\Gr$ can not be covered by a
  finite number of translated copies of $\X$, in other words, for any
  $ g_1,\ldots ,g_n\in \Gr$, we have $ \bigcup_{i=1}^n (g_i\X) \neq\Gr$.
\end{enumerate}
\begin{definition}\label{D:genss}Semi-simplicity:
  \begin{enumerate}
  \item A $\FK$-category $\cat$ is \emph{semi-simple} if all its objects are semi-simple.
  \item A $\FK$-category $\cat$ is \emph{finitely semi-simple} if it is
    semi-simple and has finitely many isomorphism classes of simple
    objects.
  \item A $\Gr$-graded 
    category $\cat$ is a \emph{generically
    $\Gr$-semi-simple} category (resp. \emph{generically finitely
    $\Gr$-semi-simple} category) if there exists a small
    symmetric subset $\X\subset \Gr$ such that for each
    $g\in\Gr\setminus\X$, $\cat_g$ is semi-simple (resp. finitely
    semi-simple).  We call $\X$ the \emph{singular locus} of $\cat$.  By a
    \emph{generic simple} object we mean a  simple object of $\cat_g$ for some $g\in\Gr\setminus\X$.  
   \end{enumerate}
\end{definition}
\begin{remark}
  For a generically $\Gr$-semi-simple category $\cat$ with singular
  locus $\X$, its ideal $\Proj$ of projective objects contains all objects of
  $\cat_g$ for $g\in\Gr\setminus\X$.   In particular, generic simple
  objects of $\cat$ are projective.
\end{remark}

The notion of generically $\Gr$-semi-simple categories appears in
\cite{GPT2,GP3} through the following generalization of fusion
categories (in particular, fusion categories satisfy the following
definition when $\Gr$ is the trivial group, $\X=\emptyset$ and
$\qd=\bb=\qdim_\cat$ is the quantum dimension): 
\begin{definition}[Relative $\Gr$-spherical category]\label{D:G-spherical} 
Let $\cat$ be a generically finitely $\Gr$-semi-simple 
  pivotal 
 $\FK$-category with
  singular locus $\X\subset\Gr$ and let $\A$ be the class of generic
  simple objects of $\cat$. We say that $\cat$ is
  \emph{$(\X,\qd)$-relative $\Gr$-spherical} if
  \begin{enumerate}
  \item \label{ID:G-sph6} there exists a map $\qd:\A\rightarrow
    \FK^{\times}$ such that $(\A,\qd)$ is a t-ambi pair,
  \item \label{ID:G-sph7} there exists a map $\bb:\A\to \FK$ such that
    $\bb(V)=\bb(V^*)$, $\bb(V)=\bb(V')$ for any isomorphic objects $V, V'\in
    \A$ and for any $g_1,g_2,g_1g_2\in \Gr\setminus\X$ and $V\in \Gr_{g_1g_2}$
    we have
  \begin{equation*}\label{eq:bb}
    \bb(V)=\sum_{V_1\in irr(\cat_{g_1}),\, V_2\in irr(\cat_{g_2})}
    \bb({V_1})\bb({V_2})\dim_\FK(\Hom_\cat(V, V_1\otimes V_2))
  \end{equation*}
  where $irr(\cat_{g_i})$ denotes a representing set of isomorphism classes
  of simple  objects of $\cat_{g_i}$.
  \end{enumerate}
\end{definition}

We finish this section by recalling the following theorem and corollary which we use later to show certain algebras are semi-simple.   For a proof of the theorem see \cite{EtGHLSVY}.  

\begin{theorem}[the Density Theorem] Let $A$ be an algebra over an algebraically closed field $k$.  
Let $\{V_i\}_i$ be a set of  irreducible pairwise nonisomorphic finite dimensional modules over $A$. Then the map $\oplus_{i}\rho_i : A \to \bigoplus_i End_k(V_i)$ is surjective.
\end{theorem}

\begin{corollary}  \label{C:DensityThm}
Let $\cat$ be the category of finite dimensional modules of the finite dimensional algebra $A$.  Let $\{V_i\}_i$ be a set of  irreducible pairwise non-isomorphic finite dimensional modules over $A$.   
If 
$$\dim(A)=\sum_i \dim(V_i)^2$$ 
  then $\cat$ is semi-simple.   
\end{corollary}

\begin{proof}
From the Density Theorem we know that $\oplus_{i}\rho_i : A \to \bigoplus_i End_k(V_i)$ is surjective.  The assumption on the dimensions implies that this map is a injective.  Thus, $A$ is isomorphic to the direct sum of matrix algebras which is semi-simple and it follows that $\cat$ is semi-simple.  
\end{proof}

\section{Extension of generic properties}\label{S:Thgen}  
 Let $\cat$ be a pivotal $\FK$-category. 
 In this
section we present two theorems which extend properties observed for
generic simple objects of $\cat$ to general properties in the full category.

\begin{proposition}\label{P:subrib}
  If $\cat$ is a braided pivotal $\FK$-category then 
  the class of objects
  $$
  \{V\in\cat:\theta_{V^*}=(\theta_V)^*\}
  $$ forms a full subcategory of
  $\cat$ which is ribbon.
\end{proposition}
\begin{proof}
  For each object $V\in\cat$, consider the automorphism $$\mathscr
  E_V=\ptr_R(c^{-1}_{V,V})\circ\ptr_R(c_{V,V})=\Ffun\bp{\epsh{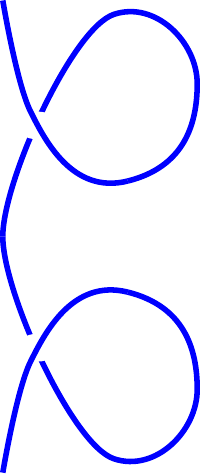}{8ex}}\in\End_\cat(V).$$
  The family $(\mathscr E_V)_{V\in\cat}$ defines a natural transformation
  (as $\mathscr E_*$ is an automorphism of the identity functor, its
  naturality is just the fact that $\mathscr E_Vf=f\mathscr E_U$ for any $f:U\to V$).  Its inverse is given by $$\mathscr E_V^{-1}
  =\Ffun\bp{\epsh{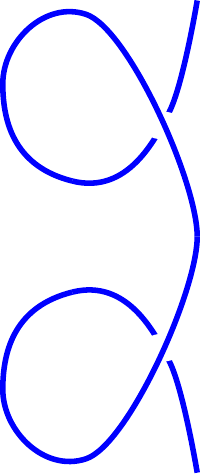}{8ex}}=\ptr_L(c^{-1}_{V,V})\circ\ptr_L(c_{V,V})
  =\ptr_L(c_{V,V})\circ\ptr_L(c^{-1}_{V,V}).$$ The naturality of
  $c_{V,V}$ implies that the image by $\Ffun$ of a $\cat$-colored diagram
  is invariant by Reidemester II and III moves so that for any
  $V,W\in\cat$
  \begin{equation}
    \label{eq:T}
    \mathscr  E_{V\otimes W}=\Ffun\bp{\epsh{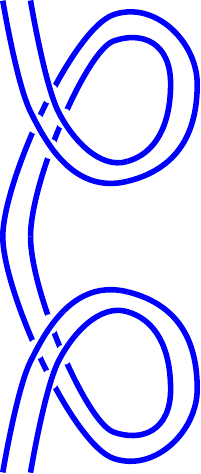}{8ex}}=
    \Ffun\bp{\epsh{fig2}{8ex}\epsh{fig2}{8ex}}=
    \mathscr  E_{V}\otimes\mathscr  E_{W}\,,
  \end{equation}
  showing that $\mathscr E_*$ is a monoidal transformation.  Finally, 
  the dual of the right partial trace is given by the left partial
  trace so
  \begin{equation}
    \label{eq:D}
    (\mathscr E_V)^*=\mathscr E_{V^*}^{-1}.
  \end{equation}
  For any object $V\in\cat$, $(\theta_V)^*=(\mathscr
  E_V)^*\circ\theta_{V^*}$ so $(\theta_V)^*=\theta_{V^*}\iff \mathscr
  E_V=\Id_V$ and the class in the proposition is clearly a sub-pivotal
  category of $\cat$ because of properties \eqref{eq:T},\eqref{eq:D}
  of the natural automorphism $\mathscr E_*$.
\end{proof}

The following theorem generalizes an argument of Ha (see \cite[Proposition 3.7]{Ha}).
\begin{theorem}\label{T:GenTwist}
  Let $\cat$ be a generically $\Gr$-semi-simple 
  pivotal 
  braided category.    
  If $\theta(V)^*=\theta(V^*)$ holds for any
  generic simple object $V$, then $\cat$ is a ribbon category
  i.e. $\theta$ is a twist on the full category $\cat$.
\end{theorem}
\begin{proof}
  Consider the natural automorphism $\mathscr E_*$ of the proof of
  Proposition \ref{P:subrib}.  By assumption we have $\mathscr
  E_V=\Id_V$ for any generic simple object $V$.  By naturality, this
  property is stable by direct sums so it is also true for any
  semi-simple object that is a direct sum of generic simple objects.
  Now for any homogenous object $W\in\cat_g$, let $h\in\Gr$ be such
  that $h,hg\notin\X$ and let $V\in\cat_h$.  Then $V\otimes
  W\in\cat_{hg}$ is semi-simple and
  $$\Id_V\otimes\Id_W=\mathscr E_{V\otimes W}=
  \mathscr E_{V}\otimes\mathscr E_{W}=\Id_V\otimes\mathscr E_{W}$$
  thus $\mathscr E_{W}=\Id_W$ and $\theta(W)^*=\theta(W^*)$.
\end{proof}

For the second theorem we first recall the relation between (right)
trace and duality:
\begin{proposition}(see \cite[Lemma 2 \& 3]{GPV}) 
  \begin{enumerate}
  \item If $\ideal$ is a right ideal then
    $\ideal^*=\{V\in\cat,\,V^*\in\ideal\}$ is a left ideal and
    $$\ideal\text{ is an ideal }\iff\ideal^*=\ideal.$$
  \item If $\qt$ is a right trace on the right ideal $\I$ then $\qt^*$
    is a left trace on $\I^*$ where by definition,
    $\qt_V^*(f)=\qt_{V^*}(f^*)$ for any $V\in\I$, any $f\in\End_\cat(V)$ and 
    $$\qt\text{ is a trace }\iff\qt^*=\qt.$$
  \end{enumerate}
\end{proposition}
\begin{definition}\ 
  \begin{enumerate}
  \item A pseudo-monoid class in a pivotal category $\cat$ is a class of objects
    stable by dual, retracts and tensor product.
  \item Let $\qt$ be a right trace on a right ideal $\I$.  The
    horizontal part of $\I$ for $\qt$ is the class
    $$\I^==\{V\in\I:V^*\in\I\text{ and }\qt_V=\qt^*_{V}\}.$$
  \end{enumerate}
\end{definition}
\begin{remark}
  An ideal of a pivotal category is a pseudo-monoid and the converse is
  partially true.  Indeed one could define the idealizer of a pseudo-monoid
  $\M$ to be the full subcategory of $\cat$ whose objects are
  $$\unit_\cat(\M)=\{V\in\cat: V\otimes \M\subset\M\supset\M\otimes V\}.$$
  Then one can prove that $\unit_\cat(\M)$ is a sub-pivotal category
  of $\cat$ which contains $\M$ such that $\M$ is an ideal in
  $\unit_\cat(\M)$.
\end{remark}

\begin{proposition}
  Let $\qt$ be a right trace on a right ideal $\I$.  Then the
  horizontal part $\I^=$ of $\I$ is a pseudo-monoid.  
  Furthermore, $\I^=$ is stable by direct sum in the following weak
  sense: if $\{V_i\}_i$ is a finite family of object of $\I^=$ and
  $W\in\I\cap\I^*$ is isomorphic to their direct sum, then $W\in\I^=$.
\end{proposition}
\begin{proof}
  If $V$ is a retract of $W\in \I^=$ then there are $\alpha:V\to
  W,\,\beta:W\to V$ with $\beta\alpha=\Id_V$.  Then $V^*$ is a retract
  of $W^*\in \I^=$ so $V^*\in\I$.   For any $f\in\End_\cat(V)$, by definition $\qt_V^*(f)=\qt_{V^*}(f^*)$ but
  $$\qt_V(f)=\qt_{V}(f\beta\alpha) =\qt_{W}(\alpha f\beta) 
  =\qt_{W^*}(\beta^*f^*\alpha^*)= \qt_{V^*}(\alpha^*\beta^*f^*)=
  \qt_{V^*}(f^*).$$ So $V\in\I^=$ and $\I^=$ is stable by retract
  (recall in particular that an isomorphic object is a retract).

  Let $V\in\I^=$ and let $\phi_V:V\stackrel\sim\to V^{**}$ be the
  pivotal isomorphism.  Then for any $f\in\End_\cat(V)$,
  $\qt_V(f)=\qt_V(\phi_V^{-1}f^{**}\phi_V)=\qt_{V^{**}}(f^{**}\phi_V\phi_V^{-1})
  =\qt_{V^{**}}(f^{**})$.  Given $g\in\End_\cat(V^*)$ then $g$ is the dual
  of $f=\phi_Vg^*\phi_V^{-1}\in\End_\cat(V)$ so
  $$\qt_{V^*}(g)=\qt_{V^*}(f^*)=\qt_V(f)=\qt_{V^{**}}(f^{**})=\qt_{V^{*}}^*(g)$$
  and $\bp{\I^=}^*\subset\I^=$.

  Let $V,W\in\I^=$ and let $f\in \End_\cat(V\otimes W)$.  Then
  $V^*\in\I$ and $\qt$ is right ambidextrous by \cite[Lemma 4]{GPV} so
  $\qt(\ptr_R(f)^*)=\qt(\ptr_L(f'))=\qt(\ptr_L(f))$ where
  $f'=(\phi_V\otimes\Id_W)f(\phi_V^{-1}\otimes\Id_W)$ so
  $$\qt(f)=\qt(\ptr_R(f))=\qt(\ptr_R(f)^*)=\qt(\ptr_L(f))
  =\qt(\ptr_L(f)^*)=\qt(\ptr_R(f^*))=\qt(f^*)$$
  thus $V\otimes W\in \I^=$.

  If $W$ is a direct sum of objects of $\I^=$ with associated
  projectors $p_i=\alpha_i\beta_i$, then so is $W^*$ with projectors
  $p_i^*$.  Now for any $f\in\End_\cat(W)$,
  $f=\bp{\sum_ip_i}f\bp{\sum_ip_i}=\sum_{i,j}p_ifp_j$ so
  $$\qt(f)=\sum_i\qt(p_ifp_i)+\sum_{i\neq j}\cancel{\qt(fp_jp_i)}
  =\sum_i\qt\bp{\beta_if\alpha_i\cancel{\beta_i \alpha_i}}
  =\sum_i\qt\bp{\alpha_i^*f^*\beta_i^*}
  =\qt(f^*).$$
\end{proof}

An object $V\in\cat$ is non-zero if $\Id_V\neq0$.
For the next theorem, we will do the following assumption that is
true for example if $\cat$ is a category of finite dimensional
representations of a Hopf $\FK$-algebra: We will assume that $\cat$ is
a Krull-Schmidt category and that for any non-zero $V\in \cat$, 
$V\otimes\cdot$ and $\cdot\otimes V$ are faithful functors.  The
former implies that any object in $\cat$ is isomorphic to an unique
direct sum of indecomposable objects (an object is indecomposable if
it is not isomorphic to a direct sum of two non-zero objects).
\begin{theorem}\label{T:GenericDimTrace}
  Let $\cat$ be a generically $\Gr$-semi-simple 
  pivotal 
   category with
  the above assumption.  Let $\qt$ be a right
  trace on $\Proj$ and let $\qd=\{\qd(V)\in \FK\}_{V\in \Proj}$ denote the associated modified dimension.   If $\qd(V)=\qd(V^*)$ holds for any
  generic simple object $V$, then $\qt$ is a trace on $\Proj$.
\end{theorem}
\begin{proof}
  By assumption we have that any generic simple object $V$ is in
  $\Proj^=$ as $\End_\cat(V)\simeq\FK\simeq\End_\cat(V^*)$ so the
  right trace on these endomorphisms 
  is determined by $\qd(V)=\qd(V^*)$.  Since  $\Proj^=$ is stable by
  direct sums it contains any semi-simple object that is a direct
  sum of generic simple objects.  Let $W$ be a indecomposable projective object.  Then there exists $g$ such that  $W\in\cat_g$.   Let $h\in\Gr$ be such that
  $h,h^{-1}g\notin\X$ (recall that $\X$ is small) and let
  $V\in\cat_{h}$.  Then $W\in \I_V=\Proj$ thus there exists $U\in\cat$
  such that $W$ is a retract of $V\otimes U$.  As
  $\cat=\bigoplus_{g\in\Gr}\cat_g$, we can assume up to replacing $U$
  by one of its summand that $U\in\cat_{h^{-1}g}$.  Then $U$ is a
  semi-simple direct sum of generic simple objects and an element of
  $\Proj^=$.  Since $\Proj^=$ is a pseudo-monoid then $V\otimes U\in\Proj^=$ and
  $W\in\Proj^=$ as $W$ is a retract.  In addition, since $\Proj^=$ is stable by direct
  sums we get that $W\in\Proj^=$ even if $W\in\Proj$ is not
  indecomposable. 
  Thus, $\Proj^==\Proj$ and $\qt=\qt^*$ on $\Proj$.  In other words, $\qt$ is a trace.
\end{proof}

\section{Quantum groups at roots of unity}\label{S:RepQG}
In this section we first recall some of the deep results established by De
Concini, Kac, Procesi, Reshetikhin and Rosso in the series of papers
\cite{DK},\cite{DKP},\cite{DKP2} and \cite{DPRR}.  
Then we observe that the twist $\theta$ and the modified dimension
$\qd$ are generically self dual.   
 As a consequence we get the existence of a ribbon structure associated to the unrolled quantum group and the existence of a trace on $\Proj$ for the categories of weight modules over the small, unrolled and unrestricted groups.

\subsection{The small, the unrolled and the unrestricted}\label{SS:Uqg(H)}
Let $\g$ be a simple finite-dimensional complex Lie algebra of rank
$\rk$ and  dimension $2N+\rk$ with the following:
\begin{enumerate}
\item  a Cartan subalgebra $\h$,
\item a root
system consisting in simple roots
$\{\alpha_1,\ldots,\alpha_\rk\}\subset\h^*$,
\item  a Cartan
matrix $A=(a_{ij})_{1\leq i,j\leq \rk}$,
\item a set $\roots^{+}$ of $N$
positive roots,
\item a root lattice $L_R=\bigoplus_i\Z\alpha_i\subset\h^*$,
\item a scalar product $\left\langle{\cdot,\cdot}\right\rangle$ on the real span of $L_R$ given by its matrix $DA=(\ang{\alpha_i,\alpha_j})_{ij}$ where $D=\diag(d_1,\ldots,d_\rk)$ and the minimum of all the $d_i$ is 1.
\end{enumerate}
The Cartan subalgebra has a basis $\{H_i\}_{i=1\cdots\rk}$ determined by
$\alpha_i(H_j)=a_{ji}$ and its dual basis of $\h^*$ is the fundamental
weights basis which generate the lattice of weights 
$L_W$.  
Let
$\rho=\frac12\sum_{\alpha\in\roots^{+}}\alpha\in \La$.
 
Let $q$ be an indeterminate and for $i=1,\ldots,\rk$, let
$q_i=q^{d_i}$.
 Let $\ell$ be an integer such that $\ell\geq 2$ (and $\ell\notin 3\Z$ if $\g=G_2$).  Let $\xi=\e^{2\sqrt{-1}\pi/\ell}$ and $r=\frac{2\ell}{3+(-1)^\ell}$.   
  For $x\in \C$ and $k,l\in \N$ we use the notation:
$$
\xi^x=\e^{\frac{2i\pi x}\ell},\quad \qn x_q=q^x-q^{-x},\quad \qN x_q=\frac{\qn
  x_q}{\qn 1_q},\quad\qN k_q!=\qN1_q\qN2_q\cdots\qN k_q,\quad{k\brack
  l}_q=\frac{\qN{k}_q!}{\qN{l}_q!\qN{k-l}_q!}.
$$    

\newcommand{\Fq}{{\mathbb K_\ell}} 
Let $\Fq$ be the subring of $\C(q)$ made of fraction that have no
poles at $\xi$ ($\Fq$ is a localization of $\C[q]$).  A $\Fq$-module
can be specialized at $q=\xi$ (the specialization is the tensor product
with the $\Fq$-module $\C$ where $q$ acts by $\xi$).

For each lattice $L$ with $L_R\subset L\subset \La$, there is an
associated quantum group which contains the group ring of $L$: 
define $\Uqg_q^L$ as the $\Fq$-algebra with generators
$K_\beta, \, X_i,\, X_{-i}$ for $\beta\in L, \,i=1,\ldots,\rk$ and
relations
\begin{eqnarray}\label{eq:rel1}
 & K_0=1, \quad K_\beta K_\gamma=K_{\beta+\gamma},\, \quad K_\beta X_{\sigma
    i}K_{-\beta}=q^{\sigma \ang{\beta,\alpha_i}}X_{\sigma i}, & \\
 \label{eq:rel2} &[X_i,X_{-j}]=
 \delta_{ij}\frac{K_{\alpha_i}-K_{\alpha_i}^{-1}}{q_i-q_i^{-1}}, &  \\
 \label{eq:rel3} & \sum_{k=0}^{1-a_{ij}}(-1)^{k}{{1-a_{ij}} \brack
   k}_{q_i} X_{\sigma i}^{k} X_{\sigma j} X_{\sigma i}^{1-a_{ij}-k} =0,
 \text{ if }i\neq j &
\end{eqnarray}
where $\sigma=\pm 1$.   
Drinfeld and Jimbo consider the quantum group corresponding to
$L_R$.  
The algebra $\Uqg_q^L$ is a Hopf algebra with coproduct
$\Delta$, counit $\epsilon$ and antipode $S$ defined by
\begin{align*}
  \Delta(X_i)&= 1\otimes X_i + X_i\otimes K_{\alpha_i}, &
  \Delta(X_{-i})&=K_{\alpha_i}^{-1} \otimes
  X_{-i} + X_{-i}\otimes 1,\\
  \Delta(K_\beta)&=K_\beta\otimes K_\beta, & \epsilon(X_i)&=
  \epsilon(X_{-i})=0, &
  \epsilon(K_{\alpha_i})&=1,\\
  S(X_i)&=-X_iK_{\alpha_i}^{-1}, & S(X_{-i})&=-K_{\alpha_i}X_{-i}, &
  S(K_\beta)&=K_{-\beta}.
\end{align*}
The \emph{unrestricted quantum group} $\Uqg_\xi^L$ is the $\C$-algebra
obtained from $\Uqg_q^L$ by specializing $q$ to $\xi$.

The \emph{unrolled quantum group} $\UqgH=\UqgH_\xi$ is the algebra generated by
$K_\beta, X_i,X_{-i},H_i$ for $\beta\in L, \,i=1,\ldots,\rk$ with Relations
\eqref{eq:rel1},\eqref{eq:rel2},\eqref{eq:rel3} where $q=\xi$ plus the relations
\begin{align}\label{eq:relH}
 [H_i,X_{\sigma j}]=\sigma a_{ij}X_{\sigma j},
  &&[H_i,H_j]=[H_i,K_\beta]=0
\end{align}
where $\sigma=\pm 1$.  The algebra $\UqgH$ is a Hopf algebra with coproduct
$\Delta$, counit $\epsilon$ and antipode $S$ defined as above on $K_\beta, \,
X_i,\, X_{-i}$ and defined on the elements $H_i$ for $i=1,\ldots,\rk$ by
\begin{align*}
  \Delta(H_i)= 1\otimes H_i + H_i\otimes 1, && \epsilon(H_i)=0, &&S(H_i)=-H_i.
\end{align*}

For a fixed choice of a convex order
$\beta_*=(\beta_1,\ldots,\beta_N)$ of $\roots^{+}$ define recursively
a convex set of root vectors $(X_{\pm\beta})_{\beta\in\beta_*}$ in
$\Uqg_\xi^L$ (see Section 8.1 and 9.1 of \cite{CP}).  The \emph{small
  quantum group} $\wb\Uqg_\xi^L$ is the quotient of $\Uqg_\xi^L$ by
the relations
$$X_{\pm\beta}^\ro=0 \text{ and } K_\gamma^{2\ro}=1\text{ for all }\beta
\in\Delta_+,\gamma\in L_R.$$
The dimension of $\wb\Uqg_\xi^L$ is $\ro^{2N}(2\ro)^n[L\!:\! L_R]$.
 From now on if it is clear we will not write the superscript $L$ in
$\Uqg_\xi^L$ or $\wb\Uqg_\xi^L$.

\subsection{Pivotal structures}\label{SS:pivotalSt}  
The element $\phi=K_{2\rho}^{1-r}$ in $U$ where $U$ is $\Uqg_\xi$, $\UqgH$ or
$\wb\Uqg_\xi$ satisfies $S^2(x)=\phi x\phi^{-1}$.  It follows that $U$ is a pivotal Hopf algebra and the category of finite dimensional $U$-modules is a pivotal
$\C$-categories, for details see \cite{Bi, GP3}.
Here the dual is given by the dual vector space $V^* =\Hom_\C(V,\C)$
equipped with the transpose action composed with antipodal morphism
$S$.  The duality morphisms $\ev$ and $\coev$ are the standard
evaluation and coevaluation whereas
$\tev=\ev\circ\tau\circ(\phi\otimes1)$ and
$\tcoev=\tau\circ(\phi^{-1}\otimes1)\circ\coev$
 where $\tau$ is the flip.  
 
 \begin{remark} \label{R:Pivotalhadic} The $h$-adic version of the
   quantum group $\Uhg=U_h$ is the $\C[[h]]$-topological Hopf algebra
   with generators $X_i,X_{-i},H_i$ for $i=1,\ldots,\rk$ and Relations
   \eqref{eq:rel2}, \eqref{eq:rel3} and \eqref{eq:relH} where $q$ and
   $K_{\alpha_i}$ are replaced by $\e^{h/2}$ and $q_i^{H_i}$,
   respectively.  The category of finite dimensional topologically
   free $U_h$-modules is a pivotal category where the pivotal element
   is normally give by $K_{2\rho}$.  In this paper we use
   $\phi=K_{2\rho}^{1-r}$ instead of $K_{2\rho}$ which also defines a
   pivotal structure.  The reason we make this choice is because the
   pivotal structure for $\phi=K_{2\rho}^{1-r}$ is compatible with the
   braiding and trace consider in future subsections.  The pivotal
   structure corresponding to $K_{2\rho}$ does not have such
   compatibilities: the left and right twist differ by a square and
   the modified dimensions of $V$ and $V^*$ are not
   equal.  
\end{remark}

\section{Quantum groups at odd ordered roots of unity}
\label{S:RepQGodd}
  
In this section we consider the case when $\ell$ is an odd integer.
In this case, $\ell=\ro$.  To make the notation clear we will write
$\cat_{odd}$, $\cat^H_{odd}$ and $\wb\cat_{odd}$ for the categories in
this section.

\subsection{Weight modules} \label{SS:weightmod} Let $Z_0$ be the
subalgebra of $\Uqg_\xi^L$ generated by
$\{X_{\pm\beta}^\ro, K_\gamma^\ro:\,\beta\in\Delta_+,\,\gamma\in L\}$.
Then $Z_0$ is a sub-Hopf algebra contained in the center of
$\Uqg_\xi^L$.  Moreover, $Z_0$ is isomorphic to the Hopf algebra of
regular functions of a group $\Gr=\Spec(Z_0)$ which is Poisson-Lie
dual to a complex simple Lie group with Lie algebra $\g$ (see
\cite{DKP}).  
Here 
$\Spec(Z_0)$ is the algebraic variety
of algebra homomorphisms from $Z_0$ to $\C$.  

By a $\Uqg_\xi$-\emph{weight module} we mean a finite dimensional
module over $\Uqg$ which restrict to a semi-simple module over $Z_0$.
Since $X_{\pm\beta}^\ro=0$ and $K_\gamma^{2\ro}=1$ in $\wb\Uqg_\xi$ we
say all finite dimensional $\wb\Uqg_\xi$-modules are \emph{weight
  modules}.  Finally, a finite dimensional $\UqgH$-module $V$ is a
$\UqgH$-\emph{weight module} if it is a semi-simple module over the
subalgebra generated by $\{H_i:i=1\cdots n\}$ and for any
$\gamma=\sum_ic_i\alpha_i\in L$ then $q^{\sum_i d_ic_i H_i}=K_\gamma$
as operators on $V$.  Let $\cat_{odd}$ (resp. $\cat^H_{odd}$,
resp. $\wb{\cat}_{odd}$) be the category of $\Uqg_\xi$ (resp. $\UqgH$,
resp. $\wb\Uqg_\xi$) weight modules.  As the action of each $K_\beta$
is determined by the action of the $H_i$'s it follows that
$\cat^H_{odd}$ does not depend on the lattice $L$.

For $g\in\Gr= \Hom_{alg}(Z_0,\C)$ let $\cat_g$ be the full subcategory
of $\cat_{odd}$ whose objects are modules where each $z\in Z_0$ act by
$g(z)$.  For $g\in \Gr$ with $g(X_{\pm\beta}^{\ro})=0$ for all
$\beta\in \Delta_+$ we say $g$ is \emph{diagonal} and the modules of
$\cat_g$ are \emph{nilpotent}.  Similarly, $\cat^H_{odd}$ is graded by
the group $\D$ of diagonal elements of $\Gr$.

The dual of a weight module is a weight module.  This implies that the category $\cat_{odd}$ (resp. $\cat^H_{odd}$, resp. $\wb{\cat}_{odd}$) is a sub-pivotal $\C$-category of the category of all finite dimensional modules.

\subsection{The braiding on $\cat^H_{odd}$}\label{SS:BraidingCatH}
In this subsection we briefly recall the existence of a
braiding for $\cat^H_{odd}$.  We
refer to Theorem 41 of \cite{GP3} for details.

Recall the $h$-adic quantum group $\Uhg=U_h$ defined in Remark
\ref{R:Pivotalhadic}.  For a root $\beta\in L_R$, let
$q_\beta=q^{\ang{\beta,\beta}/2}$.  Let
$\exp_q(x)=\sum_{i=0}^\infty
\frac{(1-q)^ix^i}{(1-q^i)(1-q^{i-1})\cdots(1-q)}$.
Consider the following elements of the $h$-adic version of the quantum
group $U_h$:
\begin{equation}
  \label{eq:Rmatrix}
  {\Hh}=q^{\sum_{i,j}d_i(A^{-1})_{ij}H_i\otimes H_j}, \;\;
  \check{R}^h=\prod_{\beta\in 
    \beta_*}\exp_{q_\beta^{-2}}\Big((q_\beta-q_\beta^{-1})X_\beta\otimes
  X_{-\beta}\Big) 
\end{equation}
where the product is ordered by the convex order $\beta_*$ of
$\roots^+$.  It is well known that $R^{h}={\Hh}\check{R}^h$ defines a
quasi-triangular structure on $U_h$ (see for example \cite{Ma}).  
\\
Let $\exp^<_q(x)=\sum_{i=0}^{r-1}
\frac{(1-q)^ix^i}{(1-q^i)(1-q^{i-1})\cdots(1-q)}$.
Replacing the occurrences of $\exp_{q_\beta^{-2}}$ with
$\exp_{q_\beta^{-2}}^<$ in the formula of $\check R ^h$ gives an
element of $\Uqg_q\otimes\Uqg_q$ (seen as a sub-algebra of
$U_h[h^{-1}]^{\otimes2}$) called the truncated quasi-R-matrix.
Specializing for $q=\xi$ we get an element
$\check R^<\in\Uqg_\xi\otimes\Uqg_\xi$.  The action of
$R^{<}={\Hh}\check{R}^<$ determines a braiding on $\cat^H$ given by
$c_{V,W}(v\otimes w)=\tau(R^{<}(v\otimes w))$ where $\tau$ is the
flip.  Summarizing we have:
\begin{proposition}\label{P:catBraidedpivotal}
The category $\cat^H_{odd}$ is a braided 
  pivotal 
category.  
\end{proposition}

\subsection{Semi-simple}\label{SS:SemiSimple}
Here we explain why $\cat_{odd}$ and $\cat_{odd}^H$ are generically
semi-simple categories.  The $g$-th graded piece of $\cat_{odd}$ can
be described using ideals of the algebra: $\cat_g$ is the category of
weight module whose annihilator contains the two sided ideal $I_g$
generated by $\{z-g(z):z\in Z_0\}$.  Hence $\cat_g$ is identified with
the category of finite dimensional modules of the finite dimensional
algebra $\Uqg_\xi/I_g$.  
For generic $g$, we will prove that the algebra $\Uqg_\xi/I_g$ is
semi-simple which will imply the category $\cat_g$ is semi-simple.

A diagonal element of $g\in\Gr$ is \emph{regular} if
$g(K_{\pm\beta}^{\ro})\neq \pm 1$ for all $\beta\in \Delta_+$.  For
regular $g$, the category $\cat_g$ has $\ro^n$ non-isomorphic highest
weight irreducible modules $V_i$.  These modules are also modules over
$\Uqg_\xi^{L_R}$ which is a subset of $\Uqg_\xi^{L}$.  Then Corollary
3.2 of \cite{DK} implies that these modules have dimension $\ro^N$.
Therefore, we have $\sum_i\dim(V_i)^2=\ro^{2N+n}$.  The PBW theorem implies $\dim(\Uqg_\xi^{L})=\sum_i\dim(V_i)^2$.  Thus, Corollary \ref{C:DensityThm} implies that for regular $g$, the category $\cat_g$ semi-simple.  
 
De Concini and Kac consider certain derivations $\underline{e}_i$ and
$\underline{f}_i$ for $i=1,...,n$.  For each $t\in\C$, in
\cite[Section 3.5]{DK}, it is shown these derivations give
automorphisms $\exp t\underline{e}_i$ and $\exp t\underline{f}_i$ of
(a completion of) 
the algebra $\Uqg_\xi$.  Let $\widetilde{G} $ be the
group generated by all these automorphisms.  This group leaves 
$Z_0$ invariant 
and acts as holomorphic transformations on the
algebraic variety $\Spec(Z_0)$.  The action of $\widetilde{G} $ on
$\Spec(Z_0)$ is called the quantum coadjoint action.  Theorem 6.1 of
\cite{DKP} considers the orbits of this action.  In particular, part
(d) of this theorem says that the union of all $\widetilde{G}$
orbits which contain at least one regular element of $\Gr$ is Zariski
open and dense in $\Gr$.  This can be reformulated as there exist a
set $\X$ with the following two properties: (1) $\Gr\setminus\X$ is a
Zariski dense open subset of $\Gr$ and (2) for each
$g\in\Gr\setminus\X$ there is a regular $d$ and an outer automorphism
of $\Uqg_\xi$ inducing an isomorphism of algebras
$\Uqg_\xi/I_g\rightarrow \Uqg_\xi/I_d$.  Therefore, for each
$g\in \Gr\setminus\X$ the algebra $\Uqg_\xi/I_g$ is semi-simple.  By
replacing $\X$ with $\X\cup \X^{-1}$ we can assume $\X$ is symmetric.
Thus, we have shown $\cat_{odd}$ is a generically finitely
$\Gr$-semi-simple category with singular locus $\X$.

\begin{remark}\label{R:regularOrbitDiag}
  By definition, every $\widetilde{G} $ orbit in $\Gr$ that contains
  an element of $\Gr\setminus\X$ also contains a regular element and thus a
  diagonal element of $\Gr$.
  \end{remark}

Recall $\cat_{odd}^H$ is graded by the group $\D$.  Let $\X_\D$ be all the non-regular elements of $\D$.  It follows from Lemma 7.1 of \cite{CGP1} that $\cat_{odd}^H$ is a generically $\D$-semi-simple category
with the singular locus $\X_\D$.  

The results described here and in Section 
\ref{S:RepQG} imply the following
proposition.  
\begin{proposition}\label{P:catHGenSSBraided}
The category $\cat_{odd}^H$ is a generically $\D$-semi-simple braided 
  pivotal 
category.  
\end{proposition}

\subsection{The ribbon structure in $\cat_{odd}^H$.}\label{SS:RibbonStr}  Let $V$ be a generic simple object of $\cat_{odd}^H$.  We will show that $\theta_{V^*}=(\theta_V)^*$.  Since $V$ is simple, the morphism $\theta_V$ is a scalar times the identity of $V$.  Moreover, $(\theta_V)^*$ is the same scalar times $\Id_{V^*}$.   The module  $V$ is determined by its highest weight which is of the form $\lambda+(\ro-1)\rho$.  Its dual $V^*$  has highest weight $-\lambda+(\ro-1)\rho$.      To compute the twist on $V$ notice that $\Hh$ is the only part of the R-matrix contributing to the following computation:
$$ \theta_V(v)=\ptr_R(c_{V,V})(v)=q^{\ang{\lambda,\lambda}-(\ro-1)^2\ang{\rho,\rho}}v  $$
where $v$ is a highest weight vector of $V$ (see \cite{GP3} for more details).  As $\ang{-\lambda,-\lambda}=\ang{\lambda,\lambda}$ then a similar calculation shows that the scalar determining $\theta_{V^*}$ is equal to the scalar determining both $\theta_V$ and $(\theta_V)^*$. 
Thus,  Proposition \ref{P:catHGenSSBraided}  and the results of this subsection  imply that $\cat_{odd}^H$ satisfies the hypothesis of Theorem \ref{T:GenTwist} and so we have:
\begin{theorem}
The category $\cat_{odd}^H$ is a ribbon category.  
 \end{theorem}
Note, in \cite{GP3} we showed that a subcategory of $\cat^H_{odd}$ is ribbon.

\subsection{The trace on the ideal of projective $\Uqg_\xi$-modules} \label{SS:modifiedDim}
 From Theorem 4.7.1 of \cite{GKP2} the ideal $\Proj$ of projective $\Uqg_\xi$-modules in $\cat_{odd}$ admits a unique nontrivial right trace which we denote by $\{\qt_V \}_{V\in \Proj}$.  In this subsection we use Theorem \ref{T:GenericDimTrace} to show this right trace is a trace.  
 
 To apply the theorem we need to show $ \qd(V)=\qd(V^*)$ for any
 generic simple $\Uqg_\xi$-module $V$.  We now explain how we reduce
 this equality to a computation related to open Hopf links.  Let $V_0$
 be the weight $\Uqg_\xi$-module determined by the highest weight
 vector $v_0$ such that $X_iv_0=0$ for $i=1,...,n$ and
 $K_\gamma v_0=q^{\sum_i d_ic_i \ang{(r-1)\rho,\alpha_i}}v_0$ for
 $\gamma=\sum_ic_i\alpha_i\in L$.  The module $V_0$ is isomorphic to
 $V_0^*$.  To simplify notation, in the following computations we
 identify these modules with an isomorphism.  Given a simple module and
 a morphism $g:W\to W$ define $\ang{g}\in \C$ as $g= \ang{g}\Id_W$.
 Suppose $f:V_0 \otimes V \to V_0 \otimes V$ is a morphism in $\cat$.
 Then
\begin{equation}\label{E:Compqd}
  \qd(V)\ang{\ptr_L(f)}  =\qt_{V}(\ptr_L(f) )=\qt_{V_0^*}( (\ptr_R(f))^*)
  =\qd(V_0^*)\ang{(\ptr_R(f))^* }= \qd(V_0)\ang{\ptr_R(f) }
\end{equation}
where the second equality is from Lemma 4(b) of \cite{GPV} and the
final equality holds because $V_0$ is isomorphic to $V_0^*$.
Therefore, to compute $\qd(V)$ we need to give a morphism $f:V_0 \otimes V \to V_0 \otimes V$ 
and compute its left and right partial trace.  We will do this now
using the braiding of $\Uhg$.

Recall the algebra $U_h$ and its R-matrix given in Subsection
\ref{SS:BraidingCatH}. Let $V_0^h$ be the simple $U_h$-weight module
with highest weight $(r-1)\rho$.  So $V_0^h$ has dimension $r^N$ and a
highest weight vector $v_0$ such that
$ H_i v_0=\ang{(r-1)\rho,\alpha_i}v_0$ for $i=1,...,n$.  The module
$V_0^h\otimes\C[h^{-1}]$ contains its $\Uqg_q$ analog
$V_0^q=\Uqg_q.v_0$ which specialize at $q=\xi$ to the
$\Uqg_\xi$-module $V_0$.

We consider the $U_h$-module $V_0^h\otimes U_h$.  
Let $W_0\subset L_R$ 
be the set of weights of $V_0^h$ and $p^\lambda$ the projector of $V_0^h$ on the weight space of weight $\lambda$.
In general, if $v$ and $w$  are vectors of weight $\lambda$ and $\mu$, respectively then  ${\Hh}(v\otimes   w)=q^{\ang{\lambda,\mu}}v\otimes w=v\otimes K_\lambda w$.  So the action of $\Hh$ on $V_0^h\otimes U_h$ is given by
\begin{equation}\label{E:HhFormula}
\Hh=\sum_{\lambda\in W_0}p^\lambda\otimes K_\lambda.
\end{equation}

Consider the endomorphism $f_{U_h}$ of $V_0^h\otimes U_h$ given by the
square of the braiding, that is $f_{U_h}$ is the multiplication by
$R_{21}R_{12}$.  Since $X^r_{\pm\alpha}$ vanishes on $V_0^h$ the only
portion of the quasi R-matrix contributing to $f_{U_h}$ is formed by
terms of the truncated quasi R-matrix.  Equations \eqref{eq:Rmatrix}
and \eqref{E:HhFormula} imply $f_{U_h}$ defines an element
$f_{\Uqg_q}\in \Aut_\Fq(V_0^q)\otimes\Uqg_q$.  Moreover, these equation imply that $q$ can be 
specialized at $\xi$  to obtain an element $f_{\Uqg_\xi}\in \Aut_\C(V_0)\otimes \Uqg_\xi$. Then for
any representation $\rho_V: \Uqg_\xi\to \End(V)$,
$f_V=(\Id\otimes\rho_V)(f_{\Uqg_\xi})$ is an automorphism of the
$\Uqg_\xi$-module $V_0\otimes V$ ($f_*$ defines a natural automorphism
of the functor $V_0\otimes*$).

\begin{lemma}\label{L:PtrRfV}
For any generic simple $\Uqg_\xi$-module $V$ we have $\ptr_R(f_V)=r^{N}\Id_{V_0}$.
\end{lemma}
\begin{proof}
Consider a highest weight vector $v_0$ of $V_0$.  
 Then Equation \eqref{E:HhFormula} implies that 
$$\check R^<(v_0\otimes x)=v_0\otimes K_{2(r-1)\rho}x+\bp{\text{terms}\in V_0'\otimes \Uqg_\xi}$$
where $x\in \Uqg_\xi$ and  $V_0'$ is the sum of weight spaces of $V_0$ of weight strictly less than $(r-1)\rho$.  
Therefore,  
$$\ptr_R(f_V)(v_0) =\ptr_R(p^{(r-1)\rho}\otimes\rho_V(K_{2(r-1)\rho}))(v_0)
=v_0\left( \sum_i x_i^*\left(K^{1-r}_{2\rho}K_{2(r-1)\rho}x_i\right)\right)
=r^{N}v_0$$
where $x_i$ is a basis of $V$ and   $x_i^*$ is its dual basis.  Since $V_0$ is simple then $\ptr_R(f_V)=r^{N}\Id_{V}$.
\end{proof}

Let $V$ be a simple module of $\cat_g$ where $g\in\Gr\setminus \X$.
Then there exists an outer automorphism $\gamma\in\wt G$ such that
$\rho_V\circ\gamma$ is a nilpotent representation of $\Uqg_\xi$ (see Remark \ref{R:regularOrbitDiag}).  We
say $\rho_V\circ\gamma$ is a {\em nilpotent deformation} of $V$.
\begin{lemma}\label{L:computingLeftptrf} 
  The element $\delta_q=\ptr_L(f_{\Uqg_q})$ belongs to the center of
  $\Uqg_q$  
 and can be specialized to obtain an element $\delta_\xi$ in the center of $\Uqg_\xi$.  
   Then for any generic simple $\Uqg_\xi$-module $V$
 we have $\ptr_L(f_{V})=\rho_V(\delta_\xi)$.  Moreover, 
  \begin{equation}
  \label{E:ptrnil}
  \ptr_L(f_{V})=\prod_{\alpha\in\roots^{+}}
  \dfrac{\qn{\ro\ang{\wtb,\alpha}}}{\qn{\ang{\wtb,\alpha}}}\Id_{V}
  \end{equation}
  where $\wtb+(r-1)\rho$ is the highest weight of a nilpotent
  deformation of $V$.
  \end{lemma}
\begin{proof}
Recall the morphism  $f_{U_h}: V_0^h\otimes U_h\to V_0^h\otimes U_h$ is given by multiplication on the left by a truncated part  of $R_{21}R_{12}$.  Using  Equations \eqref{eq:Rmatrix}
and \eqref{E:HhFormula} a direct computation shows that $\ptr_L(f_{U_h})$ is a $U_h$-endomorphism given by left multiplication of an element $\delta_h$ of  $U_h$.  Thus, $\delta_h$ is an element of the center of $U_h$.  Moreover, the direct calculation above shows $\delta_h$ is a element of $\Uqg_q$ seen as a sub algebra of $U_h[h^{-1}]$.  By definition this element is equal to $\delta_q=\ptr_L(f_{\Uqg_q})$.  Therefore, $\delta_q$ is central since $\delta_h$ is central.   

Let $Z_\xi$ be the specialization at $q=\xi$ of the center of
$\Uqg_q$.  
Then by definition of the derivations  $\underline{e}_i$ and $\underline{f}_i$ it follows that the automorphisms of $\wt G$ acts trivially on $Z_\xi$.  
Since $\delta_\xi\in Z_\xi$, it acts as the
same scalar endomorphism on $V$ and on any of its nilpotent
deformation.  Hence the general formula in Equation \eqref{E:ptrnil} can be deduce from the case when $V$ is the nilpotent module with highest weight $\wtb+(r-1)\rho$ which was computed in Proposition 45 of \cite{GP3}.  
\end{proof}

\begin{theorem}\label{T:RightTraceIsTrace}
The right trace $\{\qt_V \}_{V\in \Proj}$ on $\Proj$ in $\cat_{odd}$ is a trace.  This trace is unique up to a global scalar since this is true for the right trace.   
\end{theorem}
\begin{proof}
Let $V$ be a generic simple $\Uqg_\xi$-module.  Combining Equation \eqref{E:Compqd} and Lemmas \ref{L:PtrRfV} and \ref{L:computingLeftptrf} we have 
\begin{equation}\label{E:qdV}
\qd(V)=\frac{ \qd(V_0)\ang{\ptr_R(f) }}{\ang{\ptr_L(f)}}= \qd(V_0)\ro^N\prod_{\alpha\in\roots^{+}}
  \dfrac{\qn{\ang{\wtb,\alpha}}}{\qn{\ro\ang{\wtb,\alpha}}}
\end{equation}
where  $\wtb+(r-1)\rho$ is the highest weight of a nilpotent
  deformation of $V$.  The dual of such a nilpotent module has highest weight $-\wtb+(r-1)\rho$.  Thus, Equation \eqref{E:qdV} implies  $\qd(V)=\qd(V^*)$.  Now from Subsections \ref{SS:pivotalSt} and \ref{SS:SemiSimple}, $\cat_{odd}$ is a generically $\Gr$-semi-simple 
  pivotal 
   category so the theorem follows from Theorem \ref{T:GenericDimTrace}.  
\end{proof}

 \begin{corollary}
   There exists a unique trace (up to global scalar) on the projective
   modules of $\wb\cat_{odd}$.
\end{corollary}
\begin{proof}
  Consider the projection morphism $\Uqg_\xi\to \wb\Uqg_\xi$.  Using
  this morphism each projective $\wb\Uqg_\xi$ weight module becomes a
  projective $\Uqg_\xi$ weight module.  In fact we have
  $\wb\cat_{odd}\cong\bigoplus_{g\in \Gr_0}\cat_g$ where
  $\Gr_0=\{g\in
  \Gr:g(X^r_{\pm\beta})=0,\,g(K_\beta^r)=\pm1,\,\beta\in\Delta_+\}$.

  It follows that the trace of Theorem \ref{T:RightTraceIsTrace}
  induces a trace on the idea of projective $\wb\Uqg_\xi$-module.
  This trace is unique because from Theorem 4.7.1. of \cite{GKP2}
  there exists unique nontrivial right trace on the projective ideal
  of $\wb\Uqg_\xi$-modules.
\end{proof}

\section{The case of $\slt$}\label{SS:ccl}
\newcommand{\ve}{{\varepsilon}}
\newcommand{\vp}{{\varphi}}
\newcommand{\Ch}{{\mathsf C}}
%
In the previous subsections we considered three versions of the
quantum groups associated to any Lie algebra $\g$, associated to
different lattices $L$ when $q$ was a root of unity of odd order.  The main reason the we do not treat even ordered roots of unity above is because we use results on the quantum coadjoint action which at this point require odd order roots of unity (see \cite{DKP}).   In
this section we treat the case of 
\emph{all} orders of roots of unity for the Drinfeld-Jimbo quantization of  $\slt$ (i.e. the root lattice $L_R$).  In this case we can prove the needed results involving the quantum coadjoint action directly using the casimir element.  

Let $\ell\ge2$, $\xi=\e^{2i\pi/\ell}$ and $r=\frac{2\ell}{3+(-1)^\ell}$.
As above we consider the three versions of quantum $\slt$ associated
with the root lattice:
$$\Uqg_\xi(\slt)=\Uqg=\ang{E,F,K^{\pm1}|KE-\xi^2EK=FK-\xi^2KF=[E,F]-\frac{K-K^{-1}}{\xi-\xi^{-1}}=0},$$
$$\Uqg_\xi^H(\slt)=\Uqg^H=\ang{\Uqg,H|[H,K]=[H,E]-2E=[H,F]+2F=0} \et$$
$$\wb \Uqg=\Uqg/(E^r,F^r,K^{2r}-1).$$
As above, let $\cat_{\slt}$, $\cat_{\slt}^H$ and $\wb \cat_{\slt}$  be  their categories of weight modules, respectively.  These categories are graded as follows: 
Let $Z_0$ be the subalgebra generated by $E^r,F^r$ and $K^r$.  
The center of $\Uqg$ is generated by $Z_0$ with
the casimir 
$$\Omega=\qn1^2EF+K\xi^{-1}+K^{-1}\xi=\qn1^2FE+K\xi+K^{-1}\xi^{-1}$$
which satisfies the polynomial equation
$\Ch_r\bp{\Omega}=\qn1^{2r}E^rF^r-(-1)^\ell (K^r+K^{-r})$ where $\Ch_r$
is the renormalized $r^{th}$ Chebyshev polynomial
(determined by $\Ch_r(2\cos\theta)=2\cos(r\theta)$).   

The set 
$$\Gr=\qn{ M\bp{\kappa,\ve,\vp}=\bp{\bp{\begin{array}{cc}
      1&\ve\\0&\kappa
    \end{array}},\bp{\begin{array}{cc}
      \kappa&0\\\vp&1
    \end{array}}}:\ve,\vp\in\C,\kappa\in\C^*}$$
is a group where the multiplication is given by component wise matrix multiplication.  
Also, as above $ \Hom_{alg}(Z_0,\C)$ is a group where the multiplication is given by $g_1g_2=g_1\otimes g_2\circ\Delta$.  Then the morphism $ \Hom_{alg}(Z_0,\C)\to \Gr$ given by  
$$M\bp{\kappa,\ve,\vp}(K^r)=\kappa,\quad M\bp{\kappa,\ve,\vp}(E^r)=\qn1^{-r}\ve\et
M\bp{\kappa,\ve,\vp}(F^r)=(-1)^\ell\qn1^{-r}\vp\kappa^{-1}.
$$
is a group homomorphism.  We use this morphism to identify $\Gr$ and $ \Hom_{alg}(Z_0,\C)$.

For each
$g\in\Gr$
let $\cat_g$ be the full subcategory of $\cat_{\slt}$ whose objects are modules where each $z\in Z_0$ act by $g(z)$.  As above, for $g\in \Gr$ with $g(E^{\ro})=g(F^{\ro})=0$ we say $g$ is \emph{diagonal} and the modules of $\cat_g$ are \emph{nilpotent}.   Also, a diagonal element $g$ is \emph{regular} if $g(K^r)\neq\pm1$.  
  The category $\cat_{\slt}^H$ is graded by the group $\D \simeq(\C/2\Z,+)\simeq \h^*/L_R\simeq(\C^*,\times)$ of diagonal elements of $\Gr$. Finally $\wb \cat_{\slt}$ is
$\Z/2\Z$-graded as a full subcategory of $\cat_{\slt}$.   

As above it is known that  the element $\phi=K_{2\rho}^{1-r}$ in $\Uqg$, $\UqgH$ and $\wb\Uqg$  satisfies $S^2(x)=\phi x\phi^{-1}$.  It follows that   $\cat_{\slt}$, $\cat^H_{\slt}$ and $\wb\cat_{\slt}$ are all  pivotal $\C$-categories.  Moreover, $\UqgH$ is braided, see \cite{GP3} for the odd case and \cite{GPT1} for the even.

\begin{lemma}\label{L:catsl2GenSS}
The category $\cat_{\slt}$ is generically $\Gr$-semi-simple category with
singular locus 
$$\X=\qn{
  M\bp{\kappa,\ve,\vp}:\kappa+\dfrac1\kappa-\dfrac{\ve\vp}\kappa=\pm2}.$$ 
\end{lemma}
\begin{proof}
A direct calculations shows $\X$ is small and symmetric.  We consider the quantum coadjoint action $\wt G$ on $\Gr$, see \cite{DK} and above.  Let $B\subset G$ be the union of all $\wt G$-orbits which contain a regular element.  To prove the lemma we will show that the category $\cat_h$ is semi-simple for each $h\in B$  and that $\Gr\setminus \X\subset B$.     
To do this we follow an argument similar to the one in Subsection \ref{SS:SemiSimple}: 
 direct construction shows that  
for any regular $g$ the category $\cat_g$ has $\ro$ non-isomorphic highest weight modules of dimension $\ro$    
 (also see Corollary 3.2 of \cite{DK}).  Thus, the PBW Theorem and Corollary \ref{C:DensityThm} imply that for regular $g$, the category $\cat_g$ semi-simple. 

Now let $h= M\bp{\kappa_h,\ve,\vp}\in B$ then by definition of $B$ there exists regular $g=M\bp{\kappa_g,0,0}\in \Gr$ such that $g$ and $h$ are in the same $\wt G$-orbit.   Then the $\wt G$-action gives an outer automorphism of $\Uqg$ inducing an isomorphism of  algebras $ \Uqg/I_g\rightarrow \Uqg/I_h$ and so $\cat_h$ is semi-simple.   
Finally, the element  $\Omega_q=\qn1^2EF+Kq^{-1}+K^{-1}q$ is in the center of the $\Fq$-algebra $\Uqg_q$ and specializes to $\Omega$.   
    The quantum coadjoint action is constant on $\Omega=\Omega_\xi$ because as explained above this action is constant on the sub-algebra of $\Uqg$ corresponding to the center of $\Uqg_q$.  Therefore, if $V$ and $W$ are  modules of $\cat_g$ and $ \cat_h$, respectively then $\Ch_r\bp{\Omega}$ acts on both $V$ and $W$ by 
    $$-(-1)^\ell(\kappa_g+\kappa_g^{-1})=(-1)^\ell\ve\vp\kappa_h^{-1}-(-1)^\ell (\kappa_h+\kappa_h^{-1}).$$  
    Since $g$ is regular then $\kappa_g \neq \pm1$, so $-\ve\vp\kappa_h^{-1}+ \kappa_h+\kappa_h^{-1}\neq \pm2$.  Thus, $h\in \Gr\setminus \X$.  
    \end{proof}

\begin{corollary}\label{C:CatHRibbon}
The category $\cat_{\slt}^H$ is ribbon.  
\end{corollary}
\begin{proof}
Let $\X_H=\Z/2\Z \subset \C/2\Z$ and let $d\in \C/2\Z\setminus \X_H$.  The proof of Lemma \ref{L:catsl2GenSS} and the forgetful functor $\cat^H_{\slt}\to \cat_{\slt}$ imply $\cat_d$ has $\ro$ non-isomorphic highest weight modules of dimension $\ro$ 
(one can also see this by direct construction).    
 Then the PBW Theorem and Corollary \ref{C:DensityThm} imply that $\cat_d$ is semi-simple.  
Combining this with the results at the beginning of this subsection we have $\cat^H_{\slt}$ is generically $\Gr$-semi-simple pivotal braided category with singular locus $\X_H$.  
It is known that for a generic simple object $V$ that   $\theta(V)^*=\theta(V^*)$, see \cite{GP3} for the odd case and \cite{GPT1} for the even.  The corollary follows from Theorem~\ref{T:GenTwist}.
\end{proof}
Note when $\ell$ is ever Corollary \ref{C:CatHRibbon} was previously known, see \cite{Mur08,O, GPT1}.   When $\ell$ is odd it was known that $\cat^H_{\slt}$ contained a subcategory which is ribbon, see \cite{GP3}.  

Next we prove that $\cat_{\slt}$ has a trace and show the modified dimensions only depend on the action of the casimir element.  
First, the projective cover of the trivial module is self dual, see \cite{FGST1, FGST2} for the even case and \cite{GKP2} for the odd.  
Therefore, Corollary 3.2.1 of \cite{GKP2} implies the ideal $\Proj$ of projective $\Uqg_\xi$-modules admits a unique (up to global scalar) nontrivial right trace which we denote by $\{\qt_V \}_{V\in \Proj}$. 
We will show this right trace is a trace.  
As above let $U_h=U_h(\slt)$ be the h-adic quantum group generated by $E,F$ and $H$.   
Let $V_0^h$ be the highest weight module of $U_h$ with a highest weight vector $v_0$ such that $Hv_0=(r-1)v_0$. 
 Let  $V_0^q$ be the highest weight $\Uqg_q$-module with a highest weight vector $v$ such that $Kv_0=q^{r-1}v_0$.  Let $V_0$ be the highest weight $\Uqg$-module which is the specialization of $V_0^q$.  
Let $W_0$ be the set of weights of $V_0^h$.  Then  $W_0\subset L_R$ if $\ell$ is odd and $W_0\subset \frac12L_R$ if $\ell$ is even.  
Let $f_{U_h}$ and $f_{\Uqg_q}$ be as in Subsection \ref{SS:modifiedDim}.   As above in the odd case  
$q$ can be 
specialized at $\xi$  to obtain an element $f_{\Uqg_\xi}\in \Aut_\C(V_0)\otimes \Uqg$ and 
any for representation $\rho_V: \Uqg\to \End(V)$,
$f_V=(\Id\otimes\rho_V)(f_{\Uqg_\xi})$  is an automorphism of the
$\Uqg$-module $V_0\otimes V$.    

\begin{lemma}\label{L:S'g0}
For any generic simple $\Uqg_\xi(\slt)$-module $V$ we have $\ptr_R(f_V)=r\Id_{V_0}$.
\end{lemma}
\begin{proof}
The proof of Lemma \ref{L:PtrRfV} applies here for general $\ell$.  
\end{proof}

\begin{lemma}\label{L:S'0g}
  There exists a polynomial $P(X)\in \C[X]$ such that for any generic simple $\Uqg_\xi(\slt)$-module $V$ we have $\ptr_L(f_V)=P(\omega) \Id_V$ where
  $\omega=\ang{\rho_{V}(\Omega)}$. 
\end{lemma}
\begin{proof}
As in the proof of Lemma \ref{L:computingLeftptrf} the morphism $\ptr_L(f_{U_h})$ is given by left multiplication of a central element $\delta_h$ of $U_h(\slt)$.  The center of $U_h(\slt)$ is generated by the casimir element $\Omega_h$.  So there exists a polynomial $P(X)\in \C[X]$ such that $\delta_h=P(\Omega_h)$.
Here the partial closure is taken in $U_h(\slt)$.  On the other hand, the partial trace of the central element   $\delta_q=\ptr_L(f_{\Uqg_q})$ is taken in  $\Uqg_q(\slt)$-mod.    Recall that the pivotal structure of $U_h(\slt)$-mod and $\Uqg_q(\slt)$-mod differ by $K^{-r}$, see Remark \ref{R:Pivotalhadic}.  In particular, taking the left closer with respect to $V_0^h$ and $V_0^q$ differ by the constant $(-1)^{r-1} $ determined by $K^{-r}v_0=q^{-r(r-1)}v_0=(-1)^{r-1}v_0$.  
Therefore, $\delta_h$ viewed as an element of $\Uqg_q$ (i.e. seen as
the sub algebra of $U_h[h^{-1}]$) is equal to the element
$(-1)^{r-1}\delta_q$.  Now the morphism $\ptr_L(f_V)$ is given by left
multiplication by $\delta_\xi$ which is specialization a of $\delta_q$.
Thus, $\ptr_L(f_V)=(-1)^{r-1}P(\omega)\Id_{V}$ where
$\omega=\ang{\rho_V(\Omega)}$.
\end{proof}

\begin{corollary} \label{C:GenSimpleFormulaD} Let $V$ be a generic
  simple $\Uqg_\xi(\slt)$-module.  Then there exists
  $\alpha \in(\C\setminus \Z) \cup r\Z$ such that
  $\ang{\rho_{V}(\Omega)}=(-1)^r\bp{q^\alpha+q^{-\alpha}}$ and
 \begin{equation}\label{E:modNilCyc}
 \qd(V)=(-1)^{r-1}\prod_{j=1}^{r-1}
\frac{\qn{j}}{\qn{\alpha+r-j}}=(-1)^{r-1}\frac{r\,\qn{\alpha}}{\qn{r\alpha}}
=\frac{(-1)^{r-1}r}{q^{(1-r)\alpha}  +\cdots+q^{(r-3)\alpha}+q^{(r-1)\alpha}}.
\end{equation}
 Moreover, $ \qd(V)= \qd(V^*)$.   
 Here we have fixed the global scalar of the right trace by defining $\qd(V_0)=\qt_{V_0}(\Id_{V_0})=(-1)^{r-1}$.
\end{corollary}
\begin{proof}
From Equation \eqref{E:Compqd} we have $\qd(V)=\frac{\qd(V_0)\ang{\ptr_R(f)}}{\ang{\ptr_L(f)}}$.  So by Lemmas \ref{L:S'g0} and \ref{L:S'0g} we have $\qd(V)=\qd(W)$ for any generic simple $W$ such that $\ang{\rho_W(\Omega)}=\ang{\rho_V(\Omega)}$.   The proof of Lemma \ref{L:catsl2GenSS} implies there exists a nilpotent simple module $W$ in the quantum co-adjoint orbit of $V$.  Moreover, this proof implies that the action of $\Omega$ is the same on $V$ and $W$.  Since $W$ is nilpotent there exists $\alpha\in (\C\setminus \Z) \cup r\Z$ such that $
{\Omega}$ acts on $W$ by $(-1)^r\bp{q^\alpha+q^{-\alpha}}$.   
From \cite{CGP3} it is shown that the modified quantum dimension of $W$ is given by the formula in Equation \eqref{E:modNilCyc}.  Thus, Equation \eqref{E:modNilCyc} follows. 

To see the last statement of the corollary, notice $S(\Omega)=\Omega$ so $\ang{\rho_V(\Omega)}=\ang{\rho_{V^*}(\Omega)}$ and $\qd(V)= \qd(V^*)$. 
\end{proof}

Theorem \ref{T:GenericDimTrace}, Lemma \ref{L:catsl2GenSS} and Corollary \ref{C:GenSimpleFormulaD} imply the following corollary.  
\begin{corollary}
The right trace $\{\qt_V \}_{V\in \Proj}$ on $\Proj$ in $\cat_{\slt}$ is a trace. 
\end{corollary}

\begin{corollary}
There exists a unique (up to global scalar) trace on the projective modules of $\wb\cat_{\slt}$.
\end{corollary}

\section{Conjectures for quantum groups at even ordered roots of unity}\label{S:ConjEven}
Let $\g$ be a simple finite-dimensional complex Lie algebra with data $\h, \roots^{+}, L_R, L_W,...$ discussed in Subsection \ref{SS:Uqg(H)}.   Let $\ell$ be an even integer such that $\ell\geq 2$ (and $\ell\notin 3\Z$ if $\g=G_2$).  Let $\xi=\e^{2\sqrt{-1}\pi/\ell}$ and $r=\frac{\ell}{2}$.  As above, for each  lattice $L$, consider the three versions of the quantum groups:  $\Uqg_\xi^L(\g), \UqgH_\xi(\g) $ and $ \wb\Uqg_\xi^L(\g)$.   

The case when $\ell$ is even the situation is not very well developed at this time and needs additional work to apply the results of this paper.  For example, 
the subalgebra $Z_0$ of $\Uqg_\xi^L(\g)$  generated by 
$\{X_{\pm\beta}^\ro, K_\gamma^\ro:\,\beta\in\Delta_+,\,\gamma\in L\}$ is not necessarily commutative:  if $\ang{\beta,\alpha_i}=1$ then Equation \eqref{eq:rel1} implies
$$
K_{\beta}^rX_j=q^{r\ang{\beta,\alpha_i}}X_jK_{\beta}^{r}=-X_jK_{\beta}^{r}.
$$
So one must first define an appropriate notion of weight modules for the algebras  $\Uqg_\xi^L(\g), \UqgH_\xi(\g) $ and $ \wb\Uqg_\xi^L(\g)$. 
Also, as mentioned above the quantum co-adjoint action is not worked out for case when $\ell$ is even.  Some work in this direction has been done for $\ell\in4\Z$, see \cite{Be}.

We are lead to the following question, ``For even ordered $\xi$, can one define categories of weight modules $\cat_{even}$, $\cat^H_{even}$ and $\wb\cat_{even}$ over $\Uqg_\xi^L(\g), \UqgH_\xi(\g) $ and $ \wb\Uqg_\xi^L(\g)$, respectively such that the following conjectures are true?'' 

\begin{conjecture}\label{Conj:catHRibbon}
For even ordered $\xi$, there exists a category $\cat^H_{even}$ of modules over $\UqgH_\xi(\g)$ which is ribbon.  
 \end{conjecture}
 \begin{conjecture}
For even ordered $\xi$, there exists a category $\cat_{even}$ of weight modules over $\Uqg_\xi^L(\g)$ such that there exists a unique (up to global scalar) two sided trace on the ideal of projective modules of $\cat_{even}$.
 \end{conjecture}
 \begin{conjecture}\label{Conj:smallTrace}
For even ordered $\xi$, there exists a category $\wb\cat_{even}$ of modules over $ \wb\Uqg_\xi^L(\g)$ such that there exists a unique (up to global scalar) two sided trace on the ideal of projective modules of $\wb\cat_{even}$.
 \end{conjecture}
 
 The conjectures in this section are motivated by applications in
 low-dimensional topology.  In particular, when $\ell=4$  and $\g=\slt$ in
 \cite{BCGP} it is shown that the closed 3-manifold invariant of
 \cite{CGP1} associated to the category $\cat_{\slt}^H$ are a
 canonical normalization of Reidemeister torsion defined by Turaev
 which gives rise to a Topological Quantum Field Theory (TQFT).  It
 would be interesting to see what properties the analogous topological
 invariants have for other Lie algebras 
 at similar level.
 The first step
 in defining such invariants is a proof of Conjecture
 \ref{Conj:catHRibbon}.  Also, Conjecture \ref{Conj:smallTrace} would
 also be interesting in generalizing the work of \cite{BBG}.

\linespread{1}

\end{document}